\documentclass[11pt]{article}

\usepackage{amsmath}
\usepackage{amsthm}
\usepackage{amssymb}
\usepackage[numbers,sort&compress]{natbib}
\usepackage{mathtools}
\usepackage{booktabs}
\usepackage{placeins}
\usepackage[font=small,labelfont=bf,labelsep=period,margin=8mm,skip=6pt]{caption}
\usepackage[hidelinks,hypertexnames=false]{hyperref}

\theoremstyle{plain}
\newtheorem{theorem}{Theorem}[section]
\newtheorem{lemma}[theorem]{Lemma}
\newtheorem{corollary}[theorem]{Corollary}
\newtheorem{proposition}[theorem]{Proposition}

\theoremstyle{remark}

\newcommand{\G}{\Gamma}
\newcommand{\sq}{\operatorname{sinc}}
\newcommand{\rC}{\mathrm C}
\newcommand{\rZ}{\mathrm Z}
\newcommand{\rU}{\mathrm U}
\newcommand{\PC}{\mathcal P_{\rC}}
\newcommand{\PZ}{\mathcal P_{\rZ}}
\newcommand{\PU}{\mathcal P_{\rU}}
\newcommand{\Pc}{P_{\rC}}
\newcommand{\Pz}{P_{\rZ}}
\newcommand{\Pu}{P_{\rU}}
\newcommand{\wC}{\frac{(\frac12)_n^7(\frac14)_n(\frac34)_n}{(1)_n^9}\bigl(\tfrac1{4096}\bigr)^{n}}
\newcommand{\wZ}{\frac{(\frac12)_n^5(\frac13)_n(\frac23)_n(\frac14)_n(\frac34)_n}{(1)_n^9}\bigl(-\tfrac{27}{256}\bigr)^{n}}
\newcommand{\wU}{\frac{(1)_n^7}{(\frac12)_n^5(\frac13)_n(\frac23)_n}\bigl(\tfrac4{27}\bigr)^{n}\frac1{n^7}}
\newcommand{\file}[1]{\nolinkurl{#1}}
\newcommand{\ip}[1]{\langle\varepsilon,#1\rangle}

\title{\bf Borwein--Bradley-type identities for Ramanujan-type series for $1/\pi^4$}

\author{Noam Shalev}

\date{}

\begin{document}
\makeatletter\g@addto@macro\normalsize{\setlength\abovedisplayskip{7pt plus 2pt minus 3pt}\setlength\belowdisplayskip{7pt plus 2pt minus 3pt}}\makeatother\normalsize
\maketitle

\begin{abstract}
We prove ten conjectures of Z.-W.~Sun on Ramanujan-type series for $1/\pi^4$ and $\pi^4$ with harmonic numbers. They follow by comparing Taylor coefficients in new Borwein--Bradley-type identities, whose products depend only on $x^3$, $x^4$, $x^6$ or $x^8$. These identities come from parametric extensions of the series of Cullen and Zhao. We derive the extensions from K.~C.~Au's Wilf--Zeilberger (WZ) seeds, and their sums are trigonometric. Our proofs use Guillera's periodicity argument and give new proofs of the original series. In the same way, Dougall's ${}_5F_4$ sum gives half-integer analogues of the identities of Koecher, Borwein--Bradley and Cohen, whose Taylor coefficients are Ap\'ery-like series for products of zeta values.
\end{abstract}
\noindent\textbf{Keywords:} Ramanujan-type series, Ap\'ery-like series, WZ method, Dougall's summation, harmonic numbers, Borwein--Bradley identities

\section{Introduction and preliminaries}

Many Ap\'ery-like series for zeta values are Taylor coefficients of generating functions in one or more parameters. The summand involves finite products in the parameters, and the sum equals a simple series in them whose coefficients are the zeta values. The classical example is the identity of Koecher \cite{koecher80}, found independently by Leshchiner \cite{leshchiner81},
\[
\frac12\sum_{n\ge1}\frac{(-1)^{n+1}}{n^3\binom{2n}n}\,\frac{5n^2-x^2}{n^2-x^2}\prod_{k=1}^{n-1}\Bigl(1-\frac{x^2}{k^2}\Bigr)
=\sum_{k\ge1}\frac1{k(k^2-x^2)},
\]
which generates the values $\zeta(2s+3)$ and reduces at $x=0$ to Ap\'ery's series for $\zeta(3)$. Borwein and Bradley \cite{borweinbradley1997} found experimentally a generating function for $\zeta(4s+3)$,
\begin{equation}\label{eq:bb}
\frac52\sum_{n\ge1}\frac{(-1)^{n+1}}{\binom{2n}n}\,\frac{n}{n^4-x^4}\prod_{k=1}^{n-1}\frac{k^4+4x^4}{k^4-x^4}
=\sum_{k\ge1}\frac{k}{k^4-x^4},
\end{equation}
which Almkvist and Granville \cite{almkvistgranville1999} proved. Both are specialisations of a two-parameter identity conjectured by H.~Cohen and proved by Bradley \cite{bradley2002} and, independently, by Rivoal \cite{rivoal2004simultaneous},
\begin{equation}\label{eq:cohen}
\frac12\sum_{n\ge1}\frac{(-1)^{n+1}}{n\binom{2n}n}\,\frac{5n^2-y^2}{n^4-y^2n^2-x^4}\prod_{k=1}^{n-1}\frac{(k^2-y^2)^2+4x^4}{k^4-y^2k^2-x^4}
=\sum_{k\ge1}\frac{k}{k^4-y^2k^2-x^4},
\end{equation}
which is Koecher's identity at $x=0$ and \eqref{eq:bb} at $y=0$. Thus \eqref{eq:bb} is the restriction of a two-parameter identity to a line on which its products become products in $x^4$; we call identities whose products depend only on $x^r$ \emph{Borwein--Bradley-type identities}. Bailey, Borwein and Bradley \cite{bbb06} found
\[
3\sum_{n\ge1}\frac1{\binom{2n}n(n^2-x^2)}\prod_{k=1}^{n-1}\frac{k^2-4x^2}{k^2-x^2}=\sum_{k\ge1}\frac1{k^2-x^2}=\frac{1-\pi x\cot\pi x}{2x^2},
\]
and the Hessami Pilehroods \cite{pp08,pp11} proved and found many other identities of this kind by the Wilf--Zeilberger (WZ) method.

Ramanujan's series for $1/\pi$ \cite{ramanujan1914} have analogues for higher powers of $\pi$. The first series for $1/\pi^2$ were found and proved by Guillera \cite{guillera2002wz,guillera2003new}. For $1/\pi^4$ two rational series of this kind were known: one found by J.~Cullen in 2010 (see \cite{guillera2016bilateralsums}), and one by Y.~Zhao in 2017, who in the same MathOverflow post\footnote{\url{https://mathoverflow.net/questions/281009}} also gave an \emph{upside-down} series for $8\pi^4$, one in which the Pochhammer quotient is inverted \cite{guillerarogers2014} (see also \cite{guillera2016bilateralsums,cohenguillera2021,au2025wzseeds}). Cullen's, Zhao's and the upside-down series are, respectively,
\begin{align}
\sum_{n\ge0}\wC\,\Pc(n) &= \frac{2048}{\pi^4},\label{eq:cullen}\\
\sum_{n\ge0}\wZ\,\Pz(n) &=\frac{768}{\pi^4},\label{eq:zhao}\\
\sum_{n\ge1}\wU\,\Pu(n) &= 8\pi^4,\label{eq:upside}
\end{align}
with the polynomials
\begin{equation}\label{eq:polys}
\begin{aligned}
\Pc(n) &= 43680n^4 + 20632 n^3 + 4340n^2 +466n+21,\\
\Pz(n) &= 4528n^4  + 3180 n^3 + 972 n^2 + 147n +9,\\
\Pu(n) &= 92n^3 - 84n^2 + 27n-3.
\end{aligned}
\end{equation}
 Cohen and Guillera \cite{cohenguillera2021} survey all known formulas of this kind. Z.-W.~Sun has published many conjectures in which the summand of such a series is twisted by (generalized) harmonic numbers $H^{(r)}_n=\sum_{k=1}^n k^{-r}$ \cite{sun2022harmonicsummands,sun2024binomialI,sun2023binomialII}. All three series were proved by Au \cite{au2025wzseeds} using \emph{WZ seeds}, which are hypergeometric terms in several parameters from which a WZ pair can be generated in every lattice direction (a second proof of \eqref{eq:cullen} was given by Guillera \cite{guillera2025flawless}). Many other series of this kind are found and proven in \cite{au2026harmonic}.

\subsection*{Main results}

In this paper we find new Borwein--Bradley-type identities for \eqref{eq:cullen}--\eqref{eq:upside}, and use them to prove ten conjectures of Sun (Theorem~\ref{thm:sun}); as far as we can tell, none of them had been proved before. Each conjecture follows by comparing Taylor coefficients in one of these identities. A typical example is the sextic identity: with $\omega=e^{2\pi i/3}$ and $P_{\rC,3}(n) = 1344 n^4 + 1480 n^3  +812 n^2 + 214n + 21$, we have
\begin{multline}\label{eq:sextic}
\sum_{n\ge0}\wC\Bigl[\Pc(n)-\frac{64x^6P_{\rC,3}(n)}{(2n+1)^6}\Bigr]\frac{\prod_{k=0}^{n-1}\bigl(1-\frac{64x^6}{(2k+1)^6}\bigr)^2}{\prod_{k=1}^{2n+1}\bigl(1-\frac{x^6}{k^6}\bigr)}\\
=\frac{2048}{\pi^4}\,\frac{\prod_{k\ge0}\bigl(1-\frac{64x^6}{(2k+1)^6}\bigr)}{\prod_{k\ge1}\bigl(1-\frac{x^6}{k^6}\bigr)}
=\frac{2048}{\pi}\,x^3\cot\pi x\,\cot\pi\omega x\,\cot\pi\omega^2x.
\end{multline}
Comparing coefficients of $x^6$ in \eqref{eq:sextic} yields the following identity, conjectured by Sun \cite[Conj.~5.47(i)]{sun2023binomialII}:
\[
\sum_{n\ge0}\wC\Bigl[\Pc(n)\bigl(127H^{(6)}_{2n}-2H^{(6)}_n\bigr)+\frac{1323(4n+1)}{(2n+1)^3}\Bigr]=\frac{126976\,\pi^2}{945}.
\]

These identities come from a family of parametric extensions of \eqref{eq:cullen}--\eqref{eq:upside} (Theorems~\ref{thm:CZ} and~\ref{thm:upside}). In each of them the summand is the original one multiplied by a polynomial and finite products in the parameters, and the right-hand side is trigonometric. On special lines through the origin all these products depend only on one power $x^r$; besides \eqref{eq:sextic}, this gives identities with products in $x^3$, $x^4$ and $x^8$ (Corollary~\ref{cor:lines}). Their Taylor coefficients give infinitely many other series of Sun's type. Since no known value of any series is used, we also obtain new proofs of \eqref{eq:cullen}--\eqref{eq:upside}.

In a similar manner, from Chu and Zhang's acceleration \cite{chuzhang2014} of Dougall's ${}_5F_4$ sum we obtain half-integer analogues of the identities of Koecher, Borwein and Bradley, and Cohen, and two cubic identities (Section~\ref{sec:5f4}). The coefficients of these are Ap\'ery-like series for products of zeta values, for instance the seemingly new series
\[
\sum_{n\ge1}\frac{(-1)^{n-1}}{n^3\binom{2n}n}\Bigl(5e_3(n)+\frac{6e_2(n)}{n^3}+\frac{6e_1(n)}{n^6}+\frac6{n^9}\Bigr)=\frac23\zeta(3)^4+\frac43\zeta(3)\zeta(9),
\]
where $e_r(n)$ is the $r$-th elementary symmetric function of $1^{-3},2^{-3},\dots,(n-1)^{-3}$.

\subsection*{Method}
Au's WZ seed for Dougall's ${}_7F_6$ gives Cullen's and Zhao's series along two directions of its parameter lattice. On one hyperplane, which we call the \emph{Gamma reflection hyperplane}, every Gamma factor of Dougall's product has a partner under $z\mapsto1-z$, and so by Euler's reflection formula $\G(z)\G(1-z)=\pi/\sin\pi z$ the product becomes trigonometric (\S\ref{sec:principle}). On this hyperplane the series along the two directions become four-parameter extensions of Cullen's and Zhao's series with trigonometric sums (Theorem~\ref{thm:CZ}); for the upside-down series the same happens on a hyperplane of another seed of Au (Theorem~\ref{thm:upside}). Au notes that multi-parameter extensions of this kind exist in his framework \cite[\S3.4]{au2025wzseeds}, \cite[Remark~3.3]{au2026qwzseeds}; what is new here is their explicit trigonometric form on the Gamma reflection hyperplane and what follows from it. On this hyperplane the usual WZ summation formula \cite[Proposition~2.1]{au2025wzseeds} does not evaluate these series, because a boundary series does not vanish. Our proofs follow instead Guillera's periodicity arguments \cite{guillera2016bilateralsums,guillera2020bilateral}: a certificate from the seed makes each sum periodic in one parameter, the difference of the two sides becomes a trigonometric polynomial of low degree after multiplication by trigonometric factors, and it vanishes at the points where the series terminates.

In sections~\ref{sec:principle} and~\ref{sec:theorems} we introduce the Gamma reflection hyperplane and state the extensions. In Sections~\ref{sec:lines} and~\ref{sec:sun} we derive the Borwein--Bradley-type identities and Sun's conjectures. In Section~\ref{sec:5f4} we cover the analogous identities for Dougall's ${}_5F_4$, and in Section~\ref{sec:proofs} we prove the main theorems.

\subsection*{Further identities}
The method of this paper is not limited to the three series treated here. In experiments with other lattices and directions of Au's seeds, in particular the half lattice $n\in\mathbb N+\frac12$ and points off the Gamma reflection hyperplane, we have found numerically Borwein--Bradley-type identities for further Ramanujan-type series. Among their Taylor coefficients are several open conjectures of Sun and others, and odd-zeta analogues of the identities of Koecher and of Borwein and Bradley. We hope to return to them in future work.

\section{The Gamma reflection hyperplane of Dougall's ${}_7F_6$ seed}\label{sec:principle}

Au's seed for Dougall's ${}_7F_6$ \cite[\S3.3]{au2025wzseeds} (called \textsf{Dougall7F6} in his Mathematica package \textsf{WZSeedTools} \cite{au2026harmonic}) is
\[
f(a,b,c,d,e;m)=(-1)^{a+e}(a+2m)\,C\,\frac{\G(a+m)\,\G(1+2a-b-c-d-e+m)\prod_{y}\G(y+m)}{\G(m+1)\,\G(-a+b+c+d+e+m)\prod_{y}\G(1+a-y+m)},
\]
where $y$ runs over $\{b,c,d,e\}$ and $C$ is a quotient of Gamma functions that does not depend on $m$. We write $p=(a,b,c,d,e)$ and $f(p;m)$ for $f(a,b,c,d,e;m)$. The $m$-sum is the very-well-poised ${}_7F_6$ of Dougall's theorem, and $C$ is obtained by dividing the summand by Dougall's closed form. By Au's criterion \cite[Thm.~3.3]{au2025wzseeds}, $f$ has a WZ mate in every lattice direction $v=(A;M)\in\mathbb Z^5\times\mathbb Z$ with $M\ge0$, including the parameter directions, those with $M=0$. Cullen's series comes from $v_{\rC}=(-1,-1,-1,-1,-1;2)$ and Zhao's from $v_{\rZ}=(1,0,0,0,0;1)$ \cite[\S3.4]{au2025wzseeds}, at the base point $p=(\frac12,\dots,\frac12)$. Let
\begin{equation}\label{eq:Dougall}
D(p)=\frac{\prod_{y}\G(1+a-y)\ \prod_{y<y'<y''}\G(1+a-y-y'-y'')}{\G(1+a)\,\G(1+a-b-c-d-e)\ \prod_{y<y'}\G(1+a-y-y')}
\end{equation}
be Dougall's product for the terminating very-well-poised ${}_7F_6$ (see \cite{dougall1907}, \cite[\S4.3]{bailey1935}), continued to all parameters; here $y$ runs over $\{b,c,d,e\}$, and $y<y'$ and $y<y'<y''$ over its six pairs and four triples (for the order $b<c<d<e$). For the directions $v_{\rC}$ and $v_{\rZ}$, the series of the WZ pair, as functions of $p$, appear numerically to equal $f(p;0)\,D(p)$; Theorem~\ref{thm:CZ} proves this on the following hyperplane. The \emph{Gamma reflection hyperplane} is
\[
H=\{1+2a=b+c+d+e\}.
\]
On $H$, every Gamma factor of Dougall's product has a reflection partner. For instance, the arguments $u=1+a-e$ and $u'=1+a-b-c-d$ of two factors in the numerator of $D$ satisfy $u+u'=2+2a-b-c-d-e$, so $u+u'=1$ exactly when $1+2a=b+c+d+e$.

On $H$ we write $(b,c,d,e)=\frac12+\theta$ with $\theta\in\mathbb C^4$, so that $a=\frac12(1+\sigma)$ with $\sigma=\theta_1+\theta_2+\theta_3+\theta_4$, and we put
\begin{equation}\label{eq:phi}
\begin{aligned}
\varphi_1 &= \tfrac12(\theta_1+\theta_2+\theta_3-\theta_4),&\qquad \tau_1 &= \tfrac12(\theta_1+\theta_2+\theta_3+\theta_4),\\
\varphi_2 &= \tfrac12(\theta_1+\theta_2-\theta_3+\theta_4), & \tau_2 &= \tfrac12(\theta_1+\theta_2-\theta_3-\theta_4),\\
\varphi_3 &= \tfrac12(\theta_1-\theta_2+\theta_3+\theta_4), & \tau_3 &= \tfrac12(\theta_1-\theta_2+\theta_3-\theta_4),\\
\varphi_4 &= \tfrac12(-\theta_1+\theta_2+\theta_3+\theta_4), & \tau_4 &=\tfrac12(\theta_1-\theta_2-\theta_3+\theta_4).
\end{aligned}
\end{equation}
Thus $\varphi_1,\varphi_2,\varphi_3,-\varphi_4$ are the four half-sums $\frac12(\theta_1\pm\theta_2\pm\theta_3\pm\theta_4)$ with an odd number of minus signs, and $\tau_1,\dots,\tau_4$ the four with an even number.  The map $\theta\mapsto\varphi$ is its own inverse, and $\varphi_i=\frac\sigma2-\theta_{5-i}$. We define
\[
V \coloneqq \prod_{i=1}^4\cos\pi\theta_i,\qquad E
\coloneqq \prod_{j=1}^4\cos\pi\tau_j .
\]

\begin{lemma}\label{lem:trig}
$V-E=\sin\pi\varphi_1\,\sin\pi\varphi_2\,\sin\pi\varphi_3\,\sin\pi\varphi_4$.
\end{lemma}
\begin{proof}
By the product-to-sum formulas, with $c_\pm=\cos\pi(\theta_1\pm\theta_2)$ and $d_\pm=\cos\pi(\theta_3\pm\theta_4)$, we have $4V=(c_++c_-)(d_++d_-)$, $4E=(c_++d_+)(c_-+d_-)$ and $4\prod_i\sin\pi\varphi_i=(c_+-d_-)(d_+-c_-)$. Hence
\[
4V-4E=(c_++c_-)(d_++d_-)-(c_++d_+)(c_-+d_-)=(c_+-d_-)(d_+-c_-)=4\prod_i\sin\pi\varphi_i. \qedhere
\]
\end{proof}

In these coordinates, the sixteen Gamma factors of $D$ form eight pairs $\G(z)\G(1-z)$: in the numerator the arguments are $\{1+\varphi_i,-\varphi_i\}$, $i=1,\dots,4$, and in the denominator they are $\{1+a,-a\}$, with $a=\frac12+\tau_1$, and $\{\frac12+\tau_j,\frac12-\tau_j\}$, $j=2,3,4$. Hence, by Euler's reflection formula,
\begin{equation}\label{eq:DH}
D\big|_H=-\frac{E}{\sin\pi\varphi_1\,\sin\pi\varphi_2\,\sin\pi\varphi_3\,\sin\pi\varphi_4}.
\end{equation}

\section{Parametric extensions of the series of Cullen and Zhao}\label{sec:theorems}

In this section we state the main theorems: four-parameter extensions of the series of Cullen and Zhao on the Gamma reflection hyperplane $H$, and a three-parameter extension of the upside-down series. We state the extended Zhao identity in the coordinates $\theta$, in which $\theta=0$ gives \eqref{eq:cullen} and \eqref{eq:zhao}, and the extended Cullen identity in the coordinates $\varphi$, in which its summand is symmetric under permutations and sign changes. The extended upside-down identity comes from another seed of Au (Appendix~\ref{app:tables}), and we state it in coordinates $\ell\in\mathbb C^4$ with $\ell_1+\dots+\ell_4=0$, in which $\ell=0$ gives \eqref{eq:upside} and its summand is symmetric under permutations.

Let $e_1,\dots,e_4$ be the elementary symmetric functions of $\varphi_1^2,\dots,\varphi_4^2$, $f_1,\dots,f_4$ those of $\theta_1,\dots,\theta_4$, and $g_2,g_3,g_4$ those of $\ell_1,\dots,\ell_4$. The polynomials $\PC(n;\varphi)\in\mathbb Q(n)[e_1,\dots,e_4]$, $\PZ(n;\theta)\in\mathbb Q(n)[f_1,\dots,f_4]$ and $\PU(n;\ell)\in\mathbb Q(n)[g_2,g_3,g_4]$ of the following theorems are given in \eqref{eq:PC-full}, \eqref{eq:PZ-full} and \eqref{eq:PU-full}. At the origin they are $\Pc$, $\Pz$ and $\Pu$. We write $\sq t=\sin t/t$, with $\sq 0=1$, and $\ip\varphi=\varepsilon_1\varphi_1+\dots+\varepsilon_4\varphi_4$ for $\varepsilon\in\{\pm1\}^4$.

\begin{theorem}\label{thm:CZ}
Let $\varphi$ and $\theta$ be related by \eqref{eq:phi}, with $\varphi_i\notin\mathbb Z\setminus\{0\}$.
\begin{enumerate}
\item[(i)] (\emph{The extended Cullen identity}.) With $\varepsilon$ running over the eight vectors $(1,\pm1,\pm1,\pm1)$,
\begin{equation}\label{eq:thm-cullen}
\sum_{n\ge0}\wC\,\PC(n;\varphi)\,
\frac{\prod_{k=0}^{n-1}\prod_{\varepsilon}\Bigl(1-\frac{\ip\varphi^2}{(2k+1)^2}\Bigr)}{\prod_{k=1}^{2n+1}\prod_{i=1}^4\Bigl(1-\frac{\varphi_i^2}{k^2}\Bigr)}
=\frac{1024}{\pi^4}\,\frac{V+E}{\prod_{i=1}^4\sq(\pi \varphi_i)}.
\end{equation}
\item[(ii)] (\emph{The extended Zhao identity}.)
\begin{multline}\label{eq:thm-zhao}
\sum_{n\ge0}\wZ\,\PZ(n;\theta)\\
\times\frac{\prod_{k=0}^{n-1}\Bigl[\prod_{i=1}^4\bigl(1+\frac{2\theta_i}{2k+1}\bigr)\prod_{j=2}^4\bigl(1-\frac{4\tau_j^2}{(2k+1)^2}\bigr)\Bigr]\ \prod_{k=0}^{2n-1}\bigl(1+\frac{\sigma}{2k+1}\bigr)}
{\prod_{k=1}^{n}\prod_{i=1}^4\bigl(1-\frac{\varphi_i}{k}\bigr)\ \prod_{k=1}^{2n+1}\prod_{i=1}^4\bigl(1+\frac{\varphi_i}{k}\bigr)}
=\frac{256}{\pi^4}\,\frac{V+2E}{\prod_{i=1}^4\sq(\pi \varphi_i)} .
\end{multline}
\end{enumerate}
Both series converge absolutely and locally uniformly.
\end{theorem}

\begin{theorem}[The extended upside-down identity]\label{thm:upside}
Let $\ell\in\mathbb C^4$ with $\ell_1+\ell_2+\ell_3+\ell_4=0$ and $\ell_1+\ell_i\notin\frac12+\mathbb Z$ for $i=2,3,4$. Then
\begin{multline}\label{eq:thm-upside}
\sum_{n\ge1}\wU\,\PU(n;\ell)\,
\frac{\prod_{k=1}^{n-1}\prod_{i=1}^4\bigl(1-\frac{\ell_i}{k}\bigr)\bigl(1+\frac{2\ell_i}{k}\bigr)}{\prod_{k=0}^{n-1}\prod_{i=2}^4\bigl(1-\frac{4(\ell_1+\ell_i)^2}{(2k+1)^2}\bigr)}\\
=8\pi^4\,\frac{\prod_{k\ge1}\prod_{i=1}^4\bigl(1-\frac{\ell_i^2}{k^2}\bigr)}{\prod_{k\ge0}\prod_{i=2}^4\bigl(1-\frac{4(\ell_1+\ell_i)^2}{(2k+1)^2}\bigr)}
=8\pi^4\,\frac{\prod_{i=1}^4\sq(\pi \ell_i)}{\prod_{i=2}^4\cos\pi(\ell_1+\ell_i)} .
\end{multline}
The series converges absolutely and locally uniformly.
\end{theorem}

We prove Theorems~\ref{thm:CZ} and~\ref{thm:upside} in Section~\ref{sec:proofs}.

\section{Borwein--Bradley-type identities}\label{sec:lines}

We now restrict the identities \eqref{eq:thm-cullen}, \eqref{eq:thm-zhao} and \eqref{eq:thm-upside} to lines through the origin on which all their products are products in a single power $x^r$ of one variable $x$, as in the Borwein--Bradley identity \eqref{eq:bb}. We find these lines from the symmetries of the summands.

\begin{proposition}\label{prop:B4}
The summand and the right-hand side of \eqref{eq:thm-cullen} are unchanged under permuting $\varphi_1,\dots,\varphi_4$ and changing their signs.\footnote{We remark that these transformations form the Weyl group $W(B_4)$, of order $384$.}
\end{proposition}

\begin{proof}
The products depend only on the multisets $\{\ip\varphi^2\}$ and $\{\varphi_i^2\}$, and $\PC$ is a polynomial in $e_1,\dots,e_4$. A sign change of one $\varphi_i$ exchanges the sets $\{\pm\theta_i\}$ and $\{\pm\tau_j\}$, hence $V$ and $E$, and leaves $V+E$ and $\prod_i\sq(\pi \varphi_i)$ unchanged.
\end{proof}

In the same way, the Zhao summand is symmetric in $\theta$, since a permutation of $\theta$ permutes the numbers $\varphi_i=\sigma/2-\theta_{5-i}$ and, up to sign, $\tau_2,\tau_3,\tau_4$. The upside-down summand is symmetric in $\ell$: since $\sum\ell_i=0$, the numbers $\ell_1+\ell_i$ are, up to sign, the three sums $\ell_i+\ell_j$ with $2\le i<j$, and the denominator is even in each of them. Hence every factor of each summand is a polynomial in the $e_j$, the $f_j$ or the $g_j$.

\begin{lemma}\label{lem:collapse}
If exactly one of $e_1 , \dots, e_4$, say $e_j$, is nonzero at $v\in\mathbb C^4$, then on the line $\varphi=x\cdot v$ every factor of the Cullen summand is a polynomial in $x^{2j}$. The same holds for $f_j$ and $g_j$, with $x^j$ in place of $x^{2j}$.
\end{lemma}
\begin{proof}
$e_m$ is homogeneous of degree $2m$, so $e_m(xv)=e_m(v)x^{2m}$, which vanishes for $m\ne j$. The polynomials $f_m$ and $g_m$ have degree $m$.
\end{proof}

Such points come from roots of unity, since $\prod_{\zeta^r=1}(t-\zeta)=t^r-1$. For example, on the \emph{sextic line} $\varphi=x\cdot(0,1,\omega,\omega^2)$, with $\omega=e^{2\pi i/3}$, the symmetric polynomial $e_3=x^6$ is the only non-vanishing one. Other such lines appear in Corollary~\ref{cor:lines}. On a line with products in $x^r$ we write $\PC^{[r]}(n;x)$ and $\PZ^{[r]}(n;x)$ for the restrictions of $\PC$ and $\PZ$. They come from the terms in powers of the single surviving $e_j$ or $f_j$, and by \eqref{eq:PC-full} and \eqref{eq:PZ-full} the linear ones are
\begin{align*}
\PC(n;\varphi)&=\Pc(n)-\frac{2e_1\,P_{\rC,1}(n)}{(2n+1)^2}+\frac{4e_2\,P_{\rC,2}(n)}{(2n+1)^4}-\frac{64e_3\,P_{\rC,3}(n)}{(2n+1)^6}+\frac{64e_4\,P_{\rC,4}(n)}{(2n+1)^8}+\cdots,\\
\PZ(n;\theta)&=\Pz(n)+\frac{2f_1P_{\rZ,1}(n)}{2n+1}+\frac{8f_2P_{\rZ,2}(n)}{(2n+1)^2}-\frac{16f_3P_{\rZ,3}(n)}{(2n+1)^3}-\frac{16f_4P_{\rZ,4}(n)}{(2n+1)^4}+\cdots,
\end{align*}
with
\begin{equation}\label{eq:PCi}
\begin{alignedat}{2}
P_{\rC,1}&=29088n^4+19416n^3+5972n^2+834n+45, &\qquad P_{\rZ,1}&=1780n^4+1240n^3+441n^2+81n+6,\\
P_{\rC,2}&=7328n^4-616n^3-2188n^2-622n-43, & P_{\rZ,2}&=1456n^4+1604n^3+692n^2+129n+9,\\
P_{\rC,3}&=1344n^4+1480n^3+812n^2+214n+21, & P_{\rZ,3}&=n(512n^3+444n^2+132n+11),\\
P_{\rC,4}&=1728n^4+2152n^3+1244n^2+334n+33, & P_{\rZ,4}&=1216n^4+2068n^3+1164n^2+265n+20.
\end{alignedat}
\end{equation}

\begin{corollary}\label{cor:lines}
Let $\omega=e^{2\pi i/3}$ and $\nu=e^{i\pi/4}$. For $x\in\mathbb C$ off the poles of the right-hand sides, the following identities hold. On the sextic line $\varphi=x\cdot(0,1,\omega,\omega^2)$ the extended Cullen identity becomes \eqref{eq:sextic}. On the octic line $\varphi=x\cdot(1,\nu,\nu^2,\nu^3)$ it becomes
\begin{multline}\label{eq:octic}
\sum_{n\ge0}\wC\Bigl[\Pc(n)-\frac{64x^8P_{\rC,4}(n)}{(2n+1)^8}\Bigr]
\frac{\prod_{k=0}^{n-1}\bigl(1+\frac{2176\,x^8}{(2k+1)^8}+\frac{4096\,x^{16}}{(2k+1)^{16}}\bigr)}{\prod_{k=1}^{2n+1}\bigl(1-\frac{x^8}{k^8}\bigr)}\\
=\frac{1024}{\pi^4}\,\frac{\prod_{k\ge0}\bigl(1+\frac{48i\,x^4}{(2k+1)^4}-\frac{64x^8}{(2k+1)^8}\bigr)+\prod_{k\ge0}\bigl(1-\frac{48i\,x^4}{(2k+1)^4}-\frac{64x^8}{(2k+1)^8}\bigr)}{\prod_{k\ge1}\bigl(1-\frac{x^8}{k^8}\bigr)}\\
=\frac{2048\,x^4\,\bigl(h(x_+)h(x_-)+h(x'_+)h(x'_-)\bigr)}{\sin\pi x\,\sinh\pi x\,\bigl(\cosh\sqrt2\pi x-\cos\sqrt2\pi x\bigr)},
\end{multline}
where $x_\pm=(1\pm\nu)x$, $x'_\pm=(1\pm\nu^3)x$ and $h(t)=\frac12(\cos\pi t+\cosh\pi t)=\prod_{k\ge0}\bigl(1+\frac{4t^4}{(2k+1)^4}\bigr)$. On the cubic lines $\theta=x\cdot(0,1,\omega,\omega^2)$ of the extended Zhao identity and $\ell=x\cdot(1,\omega^2,0,\omega)$ of the extended upside-down identity,
\begin{multline}\label{eq:zhao-cubic}
\sum_{n\ge0}\wZ\,\PZ^{[3]}(n;x)\,\frac{\prod_{k=0}^{n-1}\bigl(1+\frac{8x^3}{(2k+1)^3}\bigr)\bigl(1-\frac{64x^6}{(2k+1)^6}\bigr)}{\prod_{k=1}^{n}\bigl(1+\frac{x^3}{k^3}\bigr)\prod_{k=1}^{2n+1}\bigl(1-\frac{x^3}{k^3}\bigr)}\\
=\frac{768}{\pi^4}\,\frac{\prod_{k\ge0}\bigl(1-\frac{64x^6}{(2k+1)^6}\bigr)}{\prod_{k\ge1}\bigl(1-\frac{x^6}{k^6}\bigr)}
=\frac{768}{\pi}\,x^3\cot\pi x\,\cot\pi\omega x\,\cot\pi\omega^2x,
\end{multline}
where $\PZ^{[3]}(n;x)=\Pz(n)-\frac{16x^3P_{\rZ,3}(n)}{(2n+1)^3}-\frac{64x^6(92n^2+61n+9)}{(2n+1)^4}$, and
\begin{multline}\label{eq:upside-cubic}
\sum_{n\ge1}\wU\Bigl[\Pu(n)-\frac{8x^3(5n-1)}{2n-1}\Bigr]\frac{\prod_{k=1}^{n-1}\bigl(1-\frac{x^3}{k^3}\bigr)\bigl(1+\frac{8x^3}{k^3}\bigr)}{\prod_{k=0}^{n-1}\bigl(1-\frac{64x^6}{(2k+1)^6}\bigr)}\\
=8\pi^4\,\frac{\prod_{k\ge1}\bigl(1-\frac{x^6}{k^6}\bigr)}{\prod_{k\ge0}\bigl(1-\frac{64x^6}{(2k+1)^6}\bigr)}
=\frac{8\pi}{x^3}\,\tan\pi x\,\tan\pi\omega x\,\tan\pi\omega^2 x .
\end{multline}
On the quartic lines $\varphi=x\cdot(1,i,0,0)$, $\ell=x\cdot(1,-i,i,-1)$ and $\theta=x\cdot(1,i,-1,-i)$ of the three identities,
\begin{equation}\label{eq:quartic}
\sum_{n\ge0}\wC\,\PC^{[4]}(n;x)\,\frac{\prod_{k=0}^{n-1}\bigl(1+\frac{4x^4}{(2k+1)^4}\bigr)^4}{\prod_{k=1}^{2n+1}\bigl(1-\frac{x^4}{k^4}\bigr)}
=\frac{512\,x^2(\cos\pi x+\cosh\pi x)^2}{\pi^2\sin\pi x\,\sinh\pi x},
\end{equation}
\begingroup\setlength{\multlinegap}{0pt}
\begin{multline}\label{eq:upside-quartic}
\sum_{n\ge1}\wU\Bigl[\Pu(n)-\frac{32x^4}{n}\Bigr]\frac{\prod_{k=1}^{n-1}\bigl(1-\frac{x^4}{k^4}\bigr)\bigl(1-\frac{16x^4}{k^4}\bigr)}{\prod_{k=0}^{n-1}\bigl(1+\frac{64x^4}{(2k+1)^4}\bigr)}\
=\frac{16\sin^2\pi x\,\sinh^2\pi x}{x^4(\cosh2\pi x+\cos2\pi x)},
\end{multline}
\endgroup
\begin{multline}\label{eq:zhao-quartic}
\sum_{n\ge0}\wZ\,\PZ^{[4]}(n;x)\,\frac{\prod_{k=0}^{n-1}\bigl(1-\frac{16x^4}{(2k+1)^4}\bigr)\bigl(1+\frac{64x^4}{(2k+1)^4}\bigr)}{\prod_{k=1}^{n}\bigl(1-\frac{x^4}{k^4}\bigr)\prod_{k=1}^{2n+1}\bigl(1-\frac{x^4}{k^4}\bigr)}\\
=\frac{384\,x^4(\cos2\pi x+\cosh2\pi x)}{\sin^2\pi x\,\sinh^2\pi x}-256x^4,
\end{multline}
where
\begin{align*}
\PC^{[4]}(n;x)&=\Pc(n)-\frac{4x^4P_{\rC,2}(n)}{(2n+1)^4}+\frac{16(28n+23)x^8}{(2n+1)^5}+\frac{64(4n+1)x^{12}}{(2n+1)^9},\\
\PZ^{[4]}(n;x)&=\Pz(n)+\frac{16x^4P_{\rZ,4}(n)+256x^8}{(2n+1)^4}.
\end{align*}
\end{corollary}

\begin{proof}
Each identity is the restriction of \eqref{eq:thm-cullen}, \eqref{eq:thm-zhao} or \eqref{eq:thm-upside} to its line. On each line only one elementary symmetric function is nonzero:
\begin{itemize}
\item for the extended Cullen identity, $e_3=x^6$ on the sextic line, $e_4=-x^8$ on the octic line and $e_2=-x^4$ on the quartic line;
\item for the extended Zhao identity, $f_3=x^3$ on the cubic line and $f_4=-x^4$ on the quartic line;
\item for the extended upside-down identity, $g_3=x^3$ on the cubic line and $g_4=-x^4$ on the quartic line.
\end{itemize}
By Lemma~\ref{lem:collapse}, all products are then products in a single power of $x$. We give the details for the sextic line; the other lines are similar. Since $(1-t)(1-\omega t)(1-\omega^2t)=1-t^3$, the product of the factors $1-\varphi_i^2/k^2$ is $1-x^6/k^6$. Since $\varphi_1=0$, the eight numbers $\ip\varphi$ are $0$ (twice), $\pm2x$, $\pm2\omega x$ and $\pm2\omega^2x$, so the product over $\varepsilon$ is $\bigl(1-64x^6/(2k+1)^6\bigr)^2$. No monomial of Table~\ref{tab:PC} is a power of $e_3$, so $\PC(n;\varphi)=\Pc(n)-64x^6P_{\rC,3}(n)/(2n+1)^6$. On the right-hand side, $\theta=-x\cdot(\omega^2,\omega,1,0)$ by \eqref{eq:phi}, so $V=\cos\pi x\cos\pi\omega x\cos\pi\omega^2x$, and $E=V$ by Lemma~\ref{lem:trig}, since $\sin\pi\varphi_1=0$. Moreover $\prod_i\sq(\pi\varphi_i)=\prod_{k\ge1}(1-x^6/k^6)$ by the Weierstrass product $\sin\pi z=\pi z\prod_{k\ge1}(1-z^2/k^2)$, and the analogous product $\cos\pi z=\prod_{k\ge0}\bigl(1-4z^2/(2k+1)^2\bigr)$ gives the middle expression of \eqref{eq:sextic}.
\end{proof}

By Lemma~\ref{lem:collapse} the degree $r$ obtained in this way is at most $8$, since $e_j$ have degree $2j\le8$ in $x$ and $f_j$ and $g_j$ have degree $j\le4$. We do not know of Borwein--Bradley-type identities with products of degree $10$ or higher.

\section{Proof of ten conjectures of Sun}\label{sec:sun}

We now prove Sun's conjectures by comparing Taylor coefficients on the lines of Table~\ref{tab:taylor}.

\begin{theorem}\label{thm:sun}
The following ten identities, conjectured by Sun, are true:
\begin{align}
\sum_{n\ge0}\wC\Bigl[\Pc(n)\bigl(7H^{(2)}_{2n}-2H^{(2)}_n\bigr)+3624n^2+926n+69\Bigr]&=\frac{2048}{3\pi^2},\tag{C2}\\
\sum_{n\ge0}\wC\Bigl[\Pc(n)\bigl(17H^{(4)}_{2n}-H^{(4)}_n\bigr)+\frac{1796n+193}{2n+1}\Bigr]&=\frac{8704}{45},\tag{C4}\\
\sum_{n\ge0}\wC\Bigl[\Pc(n)\bigl(127H^{(6)}_{2n}-2H^{(6)}_n\bigr)+\frac{1323(4n+1)}{(2n+1)^3}\Bigr]&=\frac{126976\,\pi^2}{945},\tag{C6}\\
\sum_{n\ge0}\wC\Bigl[\Pc(n)\bigl(4354H^{(8)}_{2n}-17H^{(8)}_n\bigr)-\frac{4182(4n+1)}{(2n+1)^5}\Bigr]&=-\frac{197632\,\pi^4}{4725},\tag{C8}
\end{align}
\begin{align}
\sum_{n\ge0}\wZ\Bigl[\Pz(n)\bigl(H_{4n}+\tfrac12H_{2n}\bigr)-(484n^3+108n^2-9n-3)\Bigr]&=0,\tag{Z1}\\
\sum_{n\ge0}\wZ\Bigl[\Pz(n)\bigl(11H^{(2)}_{2n}-4H^{(2)}_n\bigr)+1780n^2+633n+63\Bigr]&=\frac{512}{\pi^2},\tag{Z2}\\
\sum_{n\ge0}\wZ\Bigl[\Pz(n)\bigl(9H^{(3)}_{2n}-2H^{(3)}_n\bigr)-\frac{916n^2+65n-9}{2n+1}\Bigr]&=0,\tag{Z3}\\
\sum_{n\ge0}\wZ\Bigl[\Pz(n)\bigl(49H^{(4)}_{2n}-2H^{(4)}_n\bigr)+\frac{5996n^2+3071n+329}{(2n+1)^2}\Bigr]&=\frac{4096}{15},\tag{Z4}
\end{align}
\begin{align}
\sum_{n\ge1}\wU\Bigl[\Pu(n)\bigl(8H^{(2)}_{2n-1}-7H^{(2)}_{n-1}\bigr)-44n+12\Bigr]&=\frac{16\,\pi^6}{3},\tag{U2}\\
\sum_{n\ge1}\wU\Bigl[\Pu(n)\bigl(64H^{(4)}_{2n-1}+13H^{(4)}_{n-1}\bigr)+\frac{32}{n}\Bigr]&=\frac{248\,\pi^8}{45}.\tag{U4}
\end{align}
Moreover
\begin{equation}
\sum_{n\ge1}\wU\Bigl[7\,\Pu(n)H^{(3)}_{n-1}-\frac{8(5n-1)}{2n-1}\Bigr]=0. \tag{U3}
\end{equation}
\end{theorem}

(C2) and (Z1)--(Z4) are Conjectures~70(i) and~69 of \cite{sun2022harmonicsummands}, and (C4)--(C8) are Conjecture~5.47(i) of \cite{sun2023binomialII}. (U2) and (U4) are the last two identities of Conjecture~5.45 of \cite{sun2023binomialII}; the first two involve $\zeta(5)$ and are not obtained by our coefficient extractions, since the Taylor coefficients of the right-hand side of \eqref{eq:thm-upside} lie in $\pi^4\mathbb Q[\pi^2]$. (U3) is not in Sun's lists. As far as we know, none of the ten has been proved before. Hou and Sun \cite{housun2023taylor} prove other conjectures of Sun by comparing Taylor coefficients, and Patel \cite{patel2026balanced} proves conjectures of Sun on derivative sums for an upside-down analogue of Cullen's series.

\begin{proof}
Each identity follows by comparing the coefficients of $x^r$ in the source identity of Table~\ref{tab:taylor}; near $x=0$ the series converges uniformly, so on the left the coefficient can be taken termwise. In (Z1), (Z3) and (U3) the right-hand side has no term in $x^r$: that of \eqref{eq:thm-zhao} is even in $\theta$, and those of \eqref{eq:zhao-cubic} and \eqref{eq:upside-cubic} are functions of $x^6$.
\end{proof}

In the same way, higher Taylor coefficients of the Borwein--Bradley-type identities of Corollary~\ref{cor:lines}, and coefficients of other monomials in the parameters of Theorems~\ref{thm:CZ} and~\ref{thm:upside}, give many more identities of this kind, with products of harmonic numbers of several orders.

\begin{table}[t]
\caption{The sources of Theorem~\ref{thm:sun}: for each identity, the source identity in one variable $x$ and the power $x^r$ whose coefficients are compared. Four sources are Theorem~\ref{thm:CZ} or~\ref{thm:upside} restricted to the line shown; here $v\in\mathbb C^4$ is any vector with $v_1+\dots+v_4\ne0$.}\label{tab:taylor}
\centering\small\renewcommand{\arraystretch}{1.25}
\begin{tabular}{@{}llc@{}}\toprule
identity & source & coefficient\\\midrule
(C2) & \eqref{eq:thm-cullen} on $\theta=x\cdot(1,0,0,0)$ & $x^2$\\
(C4) & \eqref{eq:quartic} & $x^4$\\
(C6) & \eqref{eq:sextic} & $x^6$\\
(C8) & \eqref{eq:octic} & $x^8$\\
(Z1) & \eqref{eq:thm-zhao} on $\theta=x\,v$ & $x$\\
(Z2) & \eqref{eq:thm-zhao} on $\theta=x\cdot(0,1,-1,0)$ & $x^2$\\
(Z3) & \eqref{eq:zhao-cubic} & $x^3$\\
(Z4) & \eqref{eq:zhao-quartic} & $x^4$\\
(U2) & \eqref{eq:thm-upside} on $\ell=x\cdot(1,-1,0,0)$ & $x^2$\\
(U3) & \eqref{eq:upside-cubic} & $x^3$\\
(U4) & \eqref{eq:upside-quartic} & $x^4$\\\bottomrule
\end{tabular}
\end{table}

\section{Analogous identities for Dougall's \texorpdfstring{${}_5F_4$}{5F4}}\label{sec:5f4}

The classical identities of the introduction come from Dougall's ${}_5F_4$ sum in the same way as \eqref{eq:cullen}--\eqref{eq:upside} come from the ${}_7F_6$ sum, and here too the Gamma factors pair up by reflection at special points. Let
\[
\Phi(a;b,c,d,e)=\sum_{k\ge0}(a+2k)\,\frac{(b)_k(c)_k(d)_k(e)_k}{(1+a-b)_k(1+a-c)_k(1+a-d)_k(1+a-e)_k}.
\]
For $b=a$, Dougall's theorem gives
\[
\Phi(a;a,c,d,e)=\frac{\G(1+a-c)\G(1+a-d)\G(1+a-e)\G(1+a-c-d-e)}{\G(a)\G(1+a-c-d)\G(1+a-c-e)\G(1+a-d-e)} .
\]
Chu and Zhang \cite[Thm.~9]{chuzhang2014} rewrote $\Phi$ as a series of ratio $-\frac14$ with an explicit rational weight (see also \cite{guillera2018dougall}). At $(a;b,c,d,e)=(1+x;1+x,\frac12+x+y,\frac12+x-y,\frac12)$ the weight is a multiple of $5(2n+1)^2-4x^2-4y^2$ and Dougall's product is $(\frac12+x)(\frac14-y^2)\tan\pi x\tan\pi y/(xy)$, which gives
\begin{multline}\label{eq:oddcohen}
\sum_{n\ge0}\frac{(-1)^n\binom{2n}n}{16^n}\bigl(5(2n+1)^2-4x^2-4y^2\bigr)\frac{\prod_{k=0}^{n-1}\bigl((2k+1)^2-4(x+y)^2\bigr)\bigl((2k+1)^2-4(x-y)^2\bigr)}{\prod_{k=0}^{n}\bigl((2k+1)^2-4x^2\bigr)\bigl((2k+1)^2-4y^2\bigr)}\\
=\frac{\tan\pi x\,\tan\pi y}{2xy}.
\end{multline}
This is the half-integer analogue of Cohen's identity \eqref{eq:cohen}; at $y=0$ it reduces to an identity of Leshchiner \cite{leshchiner81}. At $y=ix$, it becomes the half-integer analogue of \eqref{eq:bb},
\begin{equation}\label{eq:oddbb}
\sum_{n\ge0}\frac{(-1)^n\binom{2n}n(2n+1)^2}{16^n}\,\frac{\prod_{k=0}^{n-1}\bigl((2k+1)^4+64x^4\bigr)}{\prod_{k=0}^{n}\bigl((2k+1)^4-16x^4\bigr)}=\frac{\tan\pi x\,\tanh\pi x}{10x^2}.
\end{equation}

With $\omega=e^{2\pi i/3}$, the points $(1+\frac x2;1+\frac x2,\frac12+\frac x2,\frac12-\frac{\omega x}2,\frac12-\frac{\omega^2x}2)$ and $(2+x;1+x,1+x,1-\omega x,1-\omega^2x)$ give
\begin{multline}\label{eq:oddcubic}
\sum_{n\ge0}\frac{(-1)^n\binom{2n}n}{16^n(2n+1)^2}\,\frac{5(2n+1)^3+x^3}{(2n+1)^3-x^3}\prod_{k=0}^{n-1}\Bigl(1+\frac{x^3}{(2k+1)^3}\Bigr)\\
=\frac{\pi^2}2\,\frac{\prod_{k\ge1}\bigl(1+\frac{x^3}{8k^3}\bigr)}{\prod_{k\ge0}\bigl(1-\frac{x^3}{(2k+1)^3}\bigr)}
=\frac{\sqrt\pi}2\prod_{\epsilon^3=1}\frac{\G(\frac12-\frac{\epsilon x}2)}{\G(1+\frac{\epsilon x}2)}
\end{multline}
and
\begin{equation}\label{eq:evencubic}
\sum_{n\ge1}\frac{(-1)^{n-1}}{n^3\binom{2n}n}\,\frac{5n^3+x^3}{n^3-x^3}\prod_{k=1}^{n-1}\Bigl(1+\frac{x^3}{k^3}\Bigr)
=\frac1{x^3}\Bigl[\prod_{k\ge1}\frac{k^3+x^3}{k^3-x^3}-1\Bigr]
=\frac1{x^3}\Bigl[\prod_{\epsilon^3=1}\frac{\G(1-\epsilon x)}{\G(1+\epsilon x)}-1\Bigr].
\end{equation}
At the second point $b=a-1$, and shifting the summation index by one gives
\[
\Phi(2+x;1+x,1+x,1-\omega x,1-\omega^2x)=\frac{(1-\omega x)(1-\omega^2x)}{x^4}\bigl(\Phi(x;x,x,-\omega x,-\omega^2x)-x\bigr),
\]
to which Dougall's theorem applies.

As in Section~\ref{sec:sun}, the coefficients of $x^3$ in \eqref{eq:oddcubic} and of $16x^4$ in \eqref{eq:oddbb} are
\begin{align}
\sum_{n\ge0}\frac{(-1)^n\binom{2n}n}{16^n(2n+1)^2}\Bigl(5\sum_{k=0}^{n}\frac1{(2k+1)^3}+\frac1{(2n+1)^3}\Bigr)&=\frac{\pi^2}2\,\zeta(3),\label{eq:sun55}\\
\sum_{n\ge0}\frac{(-1)^n\binom{2n}n}{16^n(2n+1)^2}\Bigl(5\sum_{k=0}^{n-1}\frac1{(2k+1)^4}+\frac1{(2n+1)^4}\Bigr)&=\frac{7\pi^6}{7200},\label{eq:sun59}
\end{align}
which Sun conjectured \cite[(5.5), (5.9)]{sun2010lfunctions} and Au proved \cite[Example~IV]{au2025mzv}. This gives a short new proof of both: the right-hand sides of \eqref{eq:oddcubic} and \eqref{eq:oddbb} are explicit products, so their Taylor coefficients are immediate, whereas in \cite{au2025mzv} the coefficients on the other side are coloured multiple zeta values of level $2$, which then have to be evaluated. Table~\ref{tab:5f4} lists these and the next coefficients.

\begin{table}[t]
\caption{Taylor coefficients of \eqref{eq:oddcubic}, \eqref{eq:evencubic} (in $x^3$) and \eqref{eq:oddbb} (in $t=16x^4$). Each is $\sum_n w_nB_n$, with $w_n$ the weight $\frac{(-1)^n\binom{2n}n}{16^n(2n+1)^2}$ ($n\ge0$) of the half-integer identities or $\frac{(-1)^{n-1}}{n^3\binom{2n}n}$ ($n\ge1$) of the integer one, $u=(2n+1)^{-3}$, $(2n+1)^{-4}$ or $n^{-3}$, and $e_r$ the elementary symmetric functions of the numbers $(2k+1)^{-3}$, $(2k+1)^{-4}$ ($0\le k<n$) or $k^{-3}$ ($1\le k<n$).}\label{tab:5f4}
\small\renewcommand{\arraystretch}{1.3}\setlength{\tabcolsep}{5pt}
\begin{tabular}{@{}lllll@{}}\toprule
identity & order & series term $B_n$ & value & proved by\\\midrule
\eqref{eq:oddcubic} & $x^3$ & $5e_1+6u$ & $\frac{\pi^2}2\zeta(3)$ & Au (Sun (5.5))\\
\eqref{eq:oddcubic} & $x^6$ & $5e_2+6e_1u+6u^2$ & $\frac{\pi^2}2\bigl(\frac{\zeta(3)^2}2+\frac{31}{64}\zeta(6)\bigr)$ & this paper\\
\eqref{eq:evencubic} & $x^3$ & $5e_1+6u$ & $2\zeta(3)^2$ & Chu (Sun (4.5))\\
\eqref{eq:evencubic} & $x^6$ & $5e_2+6e_1u+6u^2$ & $\frac43\zeta(3)^3+\frac23\zeta(9)$ & Chu\\
\eqref{eq:evencubic} & $x^9$ & $5e_3+6e_2u+6e_1u^2+6u^3$ & $\frac23\zeta(3)^4+\frac43\zeta(3)\zeta(9)$ & this paper\\
\eqref{eq:oddbb} & $t$ & $5e_1+u$ & $\frac{7\pi^6}{7200}$ & Au (Sun (5.9))\\
\eqref{eq:oddbb} & $t^2$ & $5e_1^2+15e_2+5e_1u+u^2$ & $\frac{181\,\pi^{10}}{18144000}$ & this paper\\
\eqref{eq:oddbb} & $t^3$ & $5e_1^3+10e_1e_2+65e_3+(5e_1^2+15e_2)u+5e_1u^2+u^3$ & $\frac{178559\,\pi^{14}}{1743565824000}$ & this paper\\\bottomrule
\end{tabular}
\end{table}

The identity \eqref{eq:evencubic} is a special case of Chu's Theorem~2.2 \cite{chu2020alternating}, and its coefficients of $x^3$ and $x^6$ are equivalent to his Examples~3.11 and~3.20; the first is Sun's conjecture \cite[(4.5)]{sun2010lfunctions}. We have not found the entries of Table~\ref{tab:5f4} marked ``this paper'', or the closed forms \eqref{eq:oddcohen}--\eqref{eq:oddcubic}, in the literature. We end with the next coefficient of \eqref{eq:evencubic}:
\begin{equation}\label{eq:evencubic9}
\sum_{n\ge1}\frac{(-1)^{n-1}}{n^3\binom{2n}n}\Bigl(5e_3(n)+\frac{6e_2(n)}{n^3}+\frac{6e_1(n)}{n^6}+\frac6{n^9}\Bigr)=\frac23\zeta(3)^4+\frac43\zeta(3)\zeta(9),
\end{equation}
where $e_r(n)$ is the $r$-th elementary symmetric function of $1^{-3},2^{-3},\dots,(n-1)^{-3}$.

\section{Proofs of the main theorems}\label{sec:proofs}

We prove the three identities in the same four steps. First, we use a certificate to show that a normalised sum is periodic in one variable. We then multiply the difference of the two sides by trigonometric factors to obtain an entire function, and a crude bound shows that this function is a trigonometric polynomial of small degree. Finally, we show that it vanishes at the points where the series terminates, which are more than its degree allows.

In each case the terms $a_n$ of the normalised sum satisfy a certificate identity
\begin{equation}\label{eq:cert}
a_n(\beta+L)-a_n(\beta)=\rho(n+1)\,a_{n+1}(\beta)-\rho(n)\,a_n(\beta),
\end{equation}
where $L$ is the period and the certificate $\rho$ is a rational function of $n$ and the parameters that vanishes at the first index of the sum. We give the three certificates $\rho_{\rC}$, $\rho_{\rZ}$ and $\rho_{\rU}$ in Appendix~\ref{app:tables}. Since $a_{n+1}/a_n$ and $a_n(\beta+L)/a_n(\beta)$ are rational functions, checking \eqref{eq:cert} is a finite computation. Summing \eqref{eq:cert} over $n<N$ and letting $N\to\infty$, we find that $\sum_na_n$ is $L$-periodic, since the terms decay geometrically.

By analytic continuation, it suffices to prove Theorem~\ref{thm:CZ} for small real $\theta_2,\theta_3,\theta_4$ and Theorem~\ref{thm:upside} for small real $\ell_1,\ell_2$. We use the following two facts.

\begin{proposition}\label{prop:count}
Let $L>0$, let $B$ be entire with $B(\beta+L)=B(\beta)$, let $K\ge0$ be an integer, and suppose that $B(s+it)=o\bigl(e^{2\pi(K+1)|t|/L}\bigr)$ as $|t|\to\infty$, uniformly in $0\le s\le L$. If $B$ vanishes at $2K+1$ points that are distinct modulo $L$, then $B=0$.
\end{proposition}
\begin{proof}
Write $B=\sum_kb_kq^k$ with $q=e^{2\pi i\beta/L}$ and $b_k=\frac1L\int_0^L B(s+it)\,e^{-2\pi ik(s+it)/L}\,ds$ for every $t$. Letting $t\to-\infty$ for $k>K$ and $t\to+\infty$ for $k<-K$ gives $b_k=0$. So $q^KB$ is a polynomial in $q$ of degree at most $2K$ with $2K+1$ distinct zeros.
\end{proof}

\begin{lemma}\label{lem:pair}
Let $y$ be real, $\gamma\ge1$, and let $a,b$ be real with $|a+k|\le\gamma|b+k|$ for $0\le k<n$. Then $|(a+i\gamma y)_n|\le\gamma^n|(b+iy)_n|$. Moreover $|(b+iy)_n|\ge|(b)_n|$.
\end{lemma}
\begin{proof}
$(a+k)^2+\gamma^2y^2\le\gamma^2\bigl((b+k)^2+y^2\bigr)$ and $(b+k)^2+y^2\ge(b+k)^2$.
\end{proof}

\subsection*{The extended Cullen and Zhao identities}
We regard everything as a function of $\beta=\theta_1=s+it$ in the strip $|s|<1.1$, with $\theta_2,\theta_3,\theta_4$ fixed in $[-\frac1{30},\frac1{30}]$. Then, by \eqref{eq:phi},
\begin{equation}\label{eq:strip}
\tau_1,\dots,\tau_4,\ \varphi_1,\ \varphi_2,\ \varphi_3,\ -\varphi_4\ \in\ \tfrac\beta2+[-0.05,0.05].
\end{equation}
In particular $|\operatorname{Re}\varphi_i|\le0.6$, and under $\beta\mapsto\beta+2$ each $\sin\pi\varphi_i$ and each $\cos\pi\tau_j$ changes sign, so $\prod_i\sin\pi\varphi_i$, $V$ and $E$ are $2$-periodic.

The $n$-th terms of the two series of Theorem~\ref{thm:CZ} are, in Pochhammer form,
\begin{gather}
T_{\rC}(n)=w_n\,\PC(n;\varphi)\,\frac{(2n+1)!^8}{(\frac12)_n^{16}}\,\frac{\prod_{\xi}(\frac12+\xi)_n}{\prod_{i=1}^4(1-\varphi_i)_{2n+1}(1+\varphi_i)_{2n+1}},\label{eq:TC-poch}\\
T_{\rZ}(n)=w^{\rZ}_n\,\PZ(n;\theta)\,\frac{n!^4(2n+1)!^4}{(\frac12)_n^{10}(\frac12)_{2n}}\,
\frac{\prod_{i}(\frac12+\theta_i)_n\prod_{j=2}^4(\frac12+\tau_j)_n(\frac12-\tau_j)_n\cdot(\frac12+\tau_1)_{2n}}{\prod_i(1-\varphi_i)_n(1+\varphi_i)_{2n+1}},\label{eq:TZ-poch}
\end{gather}
where $w_n$ and $w^{\rZ}_n$ are the weights of \eqref{eq:cullen} and \eqref{eq:zhao}, and $\xi$ runs over the sixteen numbers $\pm\theta_i,\pm\tau_j$. Put
\[
S_{\rC}=\sum_{n\ge0}\frac{T_{\rC}(n)}{\varphi_1\varphi_2\varphi_3\varphi_4},\qquad
S_{\rZ}=\sum_{n\ge0}\frac{T_{\rZ}(n)}{\varphi_1\varphi_2\varphi_3\varphi_4}.
\]
Since $\pi^4\prod_i\varphi_i\sq(\pi\varphi_i)=\prod_i\sin\pi\varphi_i$, Theorem~\ref{thm:CZ} says that
\[
B_{\rC}=\Bigl(\prod_{i=1}^4\sin\pi\varphi_i\Bigr)S_{\rC}-1024\,(V+E),\qquad
B_{\rZ}=\Bigl(\prod_{i=1}^4\sin\pi\varphi_i\Bigr)S_{\rZ}-256\,(V+2E)
\]
vanish identically.

\begin{lemma}\label{lem:CZprops}
The functions $B=B_{\rC}$ and $B=B_{\rZ}$ of $\beta$ have the following properties.
\begin{enumerate}
\item[(i)] $B$ is $2$-periodic.
\item[(ii)] $B$ is entire.
\item[(iii)] $B_{\rC}=O(|t|^8e^{2\pi|t|})$ and $B_{\rZ}=O(|t|^5e^{2\pi|t|})$ for $|t|\ge1$.
\item[(iv)] $B_{\rC}$ vanishes at $\beta=\pm\frac12$ and where $\tau_j=\frac12$, $j=1,\dots,4$; $B_{\rZ}$ vanishes at $\beta=-\frac12$, where $\tau_j=\frac12$, $j=2,3,4$, and where $\sigma=-1$. For generic $\theta_2,\theta_3,\theta_4$ these points are distinct modulo $2$.
\end{enumerate}
\end{lemma}
\begin{proof}
(i) For both sums, \eqref{eq:cert} holds with $a_n=T(n)/\prod_i\varphi_i$, $L=2$ and the certificates $\rho_{\rC}$ and $\rho_{\rZ}$ of Appendix~\ref{app:tables}, so $S_{\rC}$ and $S_{\rZ}$ are $2$-periodic, and so are the trigonometric factors.

(ii) In the strip $|s|<1.1$ we have $|\operatorname{Re}\varphi_i|\le0.6$, so the terms have poles only at $\varphi_i=0$, where $\sin\pi\varphi_i$ vanishes, and the series converge uniformly on compact sets. So $B$ is holomorphic on a strip wider than its period, hence entire.

(iii) Apply Lemma~\ref{lem:pair} by pairing the numerator symbols involving $\tau_j$ with denominator blocks of the same imaginary part ($\gamma=1$), and those involving $\theta_1$ with blocks of half its imaginary part, beginning at index $n$ ($\gamma=2$); the required inequalities between real parts follow from \eqref{eq:strip}. Bound each unused denominator factor of index $k$ below by $k+\frac25$ (as $|\operatorname{Re}\varphi_i|\le0.6$), and each $(\frac12\pm\theta_i)_n$, $i\ge2$, above by $(\frac8{15})_n$. This gives $|T_{\rC}(n)|\le|\PC(n;\varphi)|\,M_n$ and $|T_{\rZ}(n)|\le|\PZ(n;\theta)|\,M^{\rZ}_n$ for $n\ge1$, where
\begin{gather*}
M_n = \frac{4^n\,w_n\,(2n+1)!^8\,(\frac8{15})_n^6}{(\frac12)_n^{16}\,(n+\frac25)_{n+1}^6\,(2n+\frac25)^2},
\\
M^{\rZ}_n = \frac{2^n\,|w^{\rZ}_n|\,n!^4(2n+1)!^4\,(\frac8{15})_n^3}{(\frac12)_n^{10}(\frac12)_{2n}\,(2n+\frac25)^2\,(n+\frac25)_{n+1}\,(\frac25)_{2n+1}}.
\end{gather*}
These do not depend on $\beta$, and $M_{n+1}/M_n\to\frac1{64}$, $M^{\rZ}_{n+1}/M^{\rZ}_n\to\frac{27}{32}$. The polynomials satisfy $|\PC(n;\varphi)|\le C(n+1)^4(1+|t|)^{12}$ and $|\PZ(n;\theta)|\le C(n+1)^4(1+|t|)^9$, and $|\varphi_i|\ge\frac{|t|}2$, so $S_{\rC}=O(|t|^8)$ and $S_{\rZ}=O(|t|^5)$, while the trigonometric factors are $O(e^{2\pi|t|})$.

(iv) At each of these points one of the factors $(\frac12\mp\theta_1)_n$, $(\frac12\mp\tau_j)_n$, $(\frac12+\tau_1)_{2n}$ vanishes for $n\ge1$, so the series reduce to their terms $n=0$.

For $B_{\rC}$: at $\beta=\pm\frac12$ we have $V=0$, so $\prod_i\sin\pi\varphi_i=-E$ by Lemma~\ref{lem:trig}; at $\tau_j=\frac12$ we have $E=0$ and $\prod_i\sin\pi\varphi_i=V$. A direct computation, using the value of \eqref{eq:PC-full} at $n=0$ and \eqref{eq:phi}, shows that
\begin{align*}
\PC(0;\varphi)+1024\prod_i\varphi_i(1-\varphi_i^2)&\quad\text{is divisible by}\quad\prod_i(4\theta_i^2-1),\\
\PC(0;\varphi)-1024\prod_i\varphi_i(1-\varphi_i^2)&\quad\text{is divisible by}\quad\prod_j(4\tau_j^2-1).
\end{align*}
Since the series terminates, $S_{\rC}=T_{\rC}(0)/\prod_i\varphi_i=\PC(0;\varphi)/\prod_i\varphi_i(1-\varphi_i^2)$ at these points, so $S_{\rC}=-1024$ where $V=0$ and $S_{\rC}=1024$ where $E=0$. Hence $B_{\rC}=-E(S_{\rC}+1024)=0$ in the first case and $B_{\rC}=V(S_{\rC}-1024)=0$ in the second.

For $B_{\rZ}$: by Lemma~\ref{lem:trig}, $B_{\rZ}=768\bigl(R\prod_i\sin\pi\varphi_i-E\bigr)$ with $R=S_{\rZ}/768-\frac13$, and at these points, by \eqref{eq:PZ-full} at $n=0$,
\[
R=\frac{T_{\rZ}(0)}{768\,\varphi_1\varphi_2\varphi_3\varphi_4}-\frac13=-\frac{(1+\sigma)(3+\sigma)\prod_{j=2}^4(4\tau_j^2-1)}{256\prod_i\varphi_i(1+\varphi_i)},
\]
and $B_{\rZ}=0$ because:
\begin{itemize}
\item at $\beta=-\frac12$: $V=0$, so $\prod_i\sin\pi\varphi_i=-E$, and $R=-1$ by $\varphi_4(1+\varphi_4)=\frac14(1+\sigma)(3+\sigma)$ and $4\tau_j^2-1=4\varphi_{5-j}(1+\varphi_{5-j})$;
\item at $\tau_j=\frac12$, $j=2,3,4$: $E=0$ and $R=0$;
\item at $\sigma=-1$: $E=0$, since $\tau_1=\frac\sigma2$, and $R=0$.
\end{itemize}
Finally, by \eqref{eq:strip}, $c_j=\tau_j-\frac\beta2$ depends only on $\theta_2,\theta_3,\theta_4$ and $|c_j|\le0.05$, so $\tau_j=\pm\frac12$ means $\beta\equiv1-2c_j\pmod 2$; and $\sigma=-1$ means $\tau_1=-\frac12$. Hence the points are distinct modulo $2$ when the $c_j$ are distinct, that is, for generic $\theta_2,\theta_3,\theta_4$.
\end{proof}

\begin{proof}[Proof of Theorem~\ref{thm:CZ}]
Apply Proposition~\ref{prop:count} with $L=K=2$ to $B_{\rC}$ and $B_{\rZ}$, using Lemma~\ref{lem:CZprops}.
\end{proof}

\subsection*{The extended upside-down identity}
We regard everything as a function of $z=s+it$, with real $\ell_1\ne\ell_2$ fixed in $(0,\frac1{10})$, and write $\ell=(\ell_1,\ell_2,-z,z-\ell_1-\ell_2)$ and $u=\ell_1+\ell_2-z$. Let $T_{\rU}(n)$ be the $n$-th term of \eqref{eq:thm-upside}, $\psi_n=zu\,T_{\rU}(n)$ and $\Psi=\sum_{n\ge1}\psi_n$. Theorem~\ref{thm:upside} says that $\Psi=\Psi^*$, where
\[
\Psi^*=8\pi^2\,\sq(\pi \ell_1)\sq(\pi \ell_2)\,\frac{\sin\pi z\,\sin\pi u}{\cos\pi(\ell_1-z)\cos\pi(\ell_2-z)\cos\pi(\ell_1+\ell_2)},
\]
that is, that $B_{\rU}=\cos\pi(\ell_1-z)\cos\pi(\ell_2-z)\,(\Psi-\Psi^*)$ vanishes identically.

\begin{lemma}\label{lem:Uprops}
The function $B_{\rU}$ of $z$ has the following properties.
\begin{enumerate}
\item[(i)] $B_{\rU}$ is $1$-periodic.
\item[(ii)] $B_{\rU}$ is entire.
\item[(iii)] $B_{\rU}=O(e^{2\pi|t|})$ for $0\le s\le1$ and $|t|\ge1$.
\item[(iv)] $B_{\rU}$ vanishes at $z=0$, $z=\ell_1+\ell_2$ and $z=\frac12$.
\end{enumerate}
\end{lemma}
\begin{proof}
(i) \eqref{eq:cert} holds with $a_n=\psi_n$, $L=1$ and the certificate $\rho_{\rU}$ of Appendix~\ref{app:tables}, so $\Psi$ is $1$-periodic, and so is $\Psi^*$.

(ii) The poles of $\Psi^*$ and of the terms $T_{\rU}(n)$ are simple and lie in $\ell_1+\frac12+\mathbb Z$ and $\ell_2+\frac12+\mathbb Z$, where the cosines vanish. So the terms of $\cos\pi(\ell_1-z)\cos\pi(\ell_2-z)\,\Psi$ are entire, and since the series converges uniformly on compact sets away from these poles, by the maximum principle it converges uniformly on all compact sets.

(iii) In Pochhammer form the part of $T_{\rU}(n)$ that depends on $z$ is
\[
\frac{(1+z)_{n-1}\,(1-z+\ell_1+\ell_2)_{n-1}\,(1+2z-2\ell_1-2\ell_2)_{n-1}\,(1-2z)_{n-1}}{(\frac12+z-\ell_2)_n\,(\frac12-z+\ell_2)_n\,(\frac12+z-\ell_1)_n\,(\frac12-z+\ell_1)_n}.
\]
Lemma~\ref{lem:pair} applies to the factor $k$ of each numerator symbol and the factor $k+1$ of the denominator symbol below it, with $\gamma=1,1,2,2$, and the four factors $k=0$ of the denominator have modulus at least $|t|$. Hence $|T_{\rU}(n)|\le|\PU(n;\ell)|\,M^{\rU}_n/t^4$ with $M^{\rU}_{n+1}/M^{\rU}_n\to\frac{16}{27}$. As $|\PU(n;\ell)|\le C(n+1)^3(1+|t|)^2$, $\Psi=O(1)$, and the trigonometric factors are $O(e^{2\pi|t|})$.

(iv) At $z=0$ and $z=\ell_1+\ell_2$ both $\Psi$ and $\Psi^*$ vanish, because of the factors $z$ and $u$. At $z=\frac12$ the series terminates, and both sides equal $-8/(\ell_1\ell_2)$.
\end{proof}

\begin{proof}[Proof of Theorem~\ref{thm:upside}]
Apply Proposition~\ref{prop:count} with $L=K=1$ to $B_{\rU}$, using Lemma~\ref{lem:Uprops}.
\end{proof}

\appendix
\makeatletter\renewcommand{\@seccntformat}[1]{Appendix~\csname the#1\endcsname:\ }\makeatother
\section{Polynomials and certificates}\label{app:tables}

\subsection*{The polynomials of Theorems~\ref{thm:CZ} and~\ref{thm:upside}}
With $e_j$, $f_j$ and $g_j$ as in Section~\ref{sec:theorems},
\begin{multline}\label{eq:PC-full}
\PC(n;\varphi)=\Pc(n)-\frac{2e_1\,P_{\rC,1}(n)}{(2n+1)^2}+\frac{4e_2\,P_{\rC,2}(n)}{(2n+1)^4}-\frac{64e_3\,P_{\rC,3}(n)}{(2n+1)^6}+\frac{64e_4\,P_{\rC,4}(n)}{(2n+1)^8}\\
+\sum_{\deg\ge2}c_{\rC}(\kappa;n)\,e_1^{\kappa_1}e_2^{\kappa_2}e_3^{\kappa_3}e_4^{\kappa_4},
\end{multline}
with $P_{\rC,1},\dots,P_{\rC,4}$ as in \eqref{eq:PCi}, and the nineteen coefficients $c_{\rC}(\kappa;n)$ of the monomials of degree at least $2$ are listed in Table~\ref{tab:PC}. Next,
\begin{multline}\label{eq:PZ-full}
\PZ(n;\theta)=\Pz(n)+\frac{2f_1P_{\rZ,1}(n)}{2n+1}+\frac{8f_2P_{\rZ,2}(n)}{(2n+1)^2}-\frac{16f_3P_{\rZ,3}(n)}{(2n+1)^3}-\frac{16f_4P_{\rZ,4}(n)}{(2n+1)^4}\\
+\sum_{\deg\ge2}c_{\rZ}(\kappa;n)\,f_1^{\kappa_1}f_2^{\kappa_2}f_3^{\kappa_3}f_4^{\kappa_4},
\end{multline}
with $P_{\rZ,1},\dots,P_{\rZ,4}$ as in \eqref{eq:PCi}, and the forty coefficients $c_{\rZ}(\kappa;n)$ of degree at least $2$ are listed in Table~\ref{tab:PZ}. Finally,
\begin{equation}\label{eq:PU-full}
\PU(n;\ell)=\Pu(n)+(44n-12)\,g_2-\frac{8(5n-1)}{2n-1}\,g_3+\frac{32}{n}\,g_4 .
\end{equation}
\begin{table}[ht]
\caption{The coefficients $c_{\rC}(\kappa;n)$ in \eqref{eq:PC-full}.}\label{tab:PC}
\small\renewcommand{\arraystretch}{1.2}
\begin{tabular}{@{}ll@{\qquad}ll@{}}\toprule
monomial & coefficient & monomial & coefficient\\\midrule
$e_1^{2}$ & $\frac{4792n^{2}+1674n+151}{(2n+1)^{2}}$ & $e_1e_2^{2}$ & $-\frac{160(4n+1)}{(2n+1)^{7}}$\\
$e_1^{3}$ & $-\frac{4(188n+31)}{(2n+1)^{3}}$ & $e_1e_4$ & $-\frac{128(80n^{2}+46n+7)}{(2n+1)^{8}}$\\
$e_1e_2$ & $\frac{16(136n^{2}+182n+25)}{(2n+1)^{4}}$ & $e_2e_3$ & $\frac{256(8n+1)}{(2n+1)^{7}}$\\
$e_1^{4}$ & $\frac{51(4n+1)}{(2n+1)^{5}}$ & $e_1^{6}$ & $\frac{4n+1}{(2n+1)^{9}}$\\
$e_1^{2}e_2$ & $-\frac{8(116n+37)}{(2n+1)^{5}}$ & $e_1^{4}e_2$ & $-\frac{12(4n+1)}{(2n+1)^{9}}$\\
$e_1e_3$ & $\frac{128(48n^{2}+22n+3)}{(2n+1)^{6}}$ & $e_1^{2}e_2^{2}$ & $\frac{48(4n+1)}{(2n+1)^{9}}$\\
$e_2^{2}$ & $\frac{16(28n+23)}{(2n+1)^{5}}$ & $e_1^{2}e_4$ & $-\frac{64(8n+3)}{(2n+1)^{9}}$\\
$e_1^{5}$ & $-\frac{10(4n+1)}{(2n+1)^{7}}$ & $e_2^{3}$ & $-\frac{64(4n+1)}{(2n+1)^{9}}$\\
$e_1^{3}e_2$ & $\frac{80(4n+1)}{(2n+1)^{7}}$ & $e_2e_4$ & $\frac{256(8n+3)}{(2n+1)^{9}}$\\
$e_1^{2}e_3$ & $-\frac{64(8n+1)}{(2n+1)^{7}}$ & &\\
\bottomrule
\end{tabular}
\end{table}

\begin{table}[ht]
\caption{The coefficients $c_{\rZ}(\kappa;n)$ in \eqref{eq:PZ-full}.}\label{tab:PZ}
\footnotesize\renewcommand{\arraystretch}{1.15}
\begin{tabular}{@{}ll@{\quad}ll@{}}\toprule
monomial & coefficient & monomial & coefficient\\\midrule
$f_1^{2}$ & $-\frac{2(1010n^{3}+711n^{2}+159n+12)}{2n+1}$ & $f_3^{2}$ & $-\frac{64(92n^{2}+61n+9)}{(2n+1)^{4}}$\\
$f_1^{3}$ & $-\frac{2(2332n^{4}+2784n^{3}+1251n^{2}+241n+18)}{(2n+1)^{3}}$ & $f_1^{7}$ & $-\frac{4(5n+2)}{(2n+1)^{4}}$\\
$f_1f_2$ & $\frac{16(796n^{4}+924n^{3}+415n^{2}+80n+6)}{(2n+1)^{3}}$ & $f_1^{5}f_2$ & $\frac{32(5n+2)}{(2n+1)^{4}}$\\
$f_1^{4}$ & $-\frac{232n^{4}+196n^{3}+54n^{2}-5n-2}{(2n+1)^{4}}$ & $f_1^{4}f_3$ & $-\frac{24(17n+6)}{(2n+1)^{4}}$\\
$f_1^{2}f_2$ & $-\frac{4(752n^{4}+1196n^{3}+648n^{2}+155n+14)}{(2n+1)^{4}}$ & $f_1^{3}f_2^{2}$ & $-\frac{64(5n+2)}{(2n+1)^{4}}$\\
$f_1f_3$ & $-\frac{8(432n^{4}+372n^{3}+96n^{2}-3n-2)}{(2n+1)^{4}}$ & $f_1^{3}f_4$ & $-\frac{32(13n^{2}+9n+2)}{(2n+1)^{5}}$\\
$f_2^{2}$ & $\frac{16(172n^{2}+75n+9)}{(2n+1)^{2}}$ & $f_1^{2}f_2f_3$ & $\frac{96(17n+6)}{(2n+1)^{4}}$\\
$f_1^{5}$ & $\frac{2(244n^{3}+228n^{2}+75n+10)}{(2n+1)^{4}}$ & $f_1f_2f_4$ & $\frac{128(13n^{2}+9n+2)}{(2n+1)^{5}}$\\
$f_1^{3}f_2$ & $-\frac{8(438n^{3}+411n^{2}+130n+16)}{(2n+1)^{4}}$ & $f_1f_3^{2}$ & $-\frac{64(31n+10)}{(2n+1)^{4}}$\\
$f_1^{2}f_3$ & $\frac{16(94n^{3}+75n^{2}+24n+4)}{(2n+1)^{4}}$ & $f_3f_4$ & $-\frac{128(16n^{2}+9n+2)}{(2n+1)^{5}}$\\
$f_1f_2^{2}$ & $\frac{32(97n^{2}+43n+6)}{(2n+1)^{3}}$ & $f_1^{8}$ & $-\frac{2}{(2n+1)^{4}}$\\
$f_1f_4$ & $-\frac{32(310n^{3}+351n^{2}+132n+16)}{(2n+1)^{4}}$ & $f_1^{6}f_2$ & $\frac{16}{(2n+1)^{4}}$\\
$f_2f_3$ & $-\frac{64n(32n+5)}{(2n+1)^{3}}$ & $f_1^{5}f_3$ & $-\frac{32}{(2n+1)^{4}}$\\
$f_1^{6}$ & $\frac{4n(4n-1)}{(2n+1)^{4}}$ & $f_1^{4}f_2^{2}$ & $-\frac{32}{(2n+1)^{4}}$\\
$f_1^{4}f_2$ & $-\frac{4(22n^{2}-13n-2)}{(2n+1)^{4}}$ & $f_1^{4}f_4$ & $-\frac{16(7n+2)}{(2n+1)^{5}}$\\
$f_1^{3}f_3$ & $-\frac{16(49n^{2}+41n+6)}{(2n+1)^{4}}$ & $f_1^{3}f_2f_3$ & $\frac{128}{(2n+1)^{4}}$\\
$f_1^{2}f_2^{2}$ & $\frac{16(6n^{2}-9n-2)}{(2n+1)^{4}}$ & $f_1^{2}f_2f_4$ & $\frac{64(7n+2)}{(2n+1)^{5}}$\\
$f_1^{2}f_4$ & $-\frac{96(17n^{2}+11n+2)}{(2n+1)^{4}}$ & $f_1^{2}f_3^{2}$ & $-\frac{128}{(2n+1)^{4}}$\\
$f_1f_2f_3$ & $\frac{32(108n^{2}+87n+14)}{(2n+1)^{4}}$ & $f_1f_3f_4$ & $-\frac{128(7n+2)}{(2n+1)^{5}}$\\
$f_2f_4$ & $\frac{64(32n^{2}+19n+4)}{(2n+1)^{4}}$ & $f_4^{2}$ & $\frac{256}{(2n+1)^{4}}$\\
\bottomrule
\end{tabular}
\end{table}

\subsection*{The certificates}
We computed the certificates of \eqref{eq:cert} with \textsf{WZSeedTools}, as WZ mates of parameter steps of the seeds. They are as follows; we write $\beta=\theta_1$ and $u=\ell_1+\ell_2-z$. For the extended Cullen identity, with $\chi$ running over $\varphi_1,\varphi_2,\varphi_3,-\varphi_4$,
\[
\rho_{\rC}(n)=\frac{16384\,n\,N_{\rC}(n)\prod_\chi(2n-\chi)(2n+1-\chi)(2n+1+\chi)}{(2n+1)^8\,\PC(n;\varphi)\,(1+2\beta-2n)(3+2\beta-2n)\prod_{j=1}^4(1+2\tau_j-2n)},
\]
where, with $s_1,s_2$ the elementary symmetric functions of $\theta_2^2,\theta_3^2,\theta_4^2$ and $k_3=\theta_2\theta_3\theta_4$,
\begin{align*}
N_{\rC}(n)={}&2+12\beta+27\beta^2+29\beta^3+15\beta^4+3\beta^5-s_1(1+5\beta+6\beta^2+2\beta^3)-(1+\beta)(s_1^2-4s_2)-2k_3\\
&+8n\bigl((1+\beta)^3-(1+\beta)s_1+2k_3\bigr)+n^2\bigl(20(1+\beta)s_1-24k_3-24-64\beta-60\beta^2-20\beta^3\bigr)\\
&+32(1+\beta)n^3-64(1+\beta)n^4.
\end{align*}
For the extended Zhao identity,
\[
\rho_{\rZ}(n)=-\frac{2048\,n\,N_{\rZ}(n)\,(\varphi_4+2n)(\varphi_4+2n+1)\prod_{i=1}^3(\varphi_i-n)(\varphi_i+2n+1)}{(2n+1)^4\,\PZ(n;\theta)\,(1+2\beta)(3+2\beta)(1+2\tau_1)\prod_{j=2}^4(1+2\tau_j-2n)},
\]
where, with $p_1,p_2,p_3$ the elementary symmetric functions of $\theta_2,\theta_3,\theta_4$,
\begin{align*}
N_{\rZ}(n)={}&6\beta^5+30\beta^4-2\beta^3(2p_1^2-4p_2-29)-6\beta^2(2p_1^2-4p_2-9)\\
&-2\beta(p_1^4+5p_1^2-4p_1^2p_2-10p_2+8p_1p_3-12)-2(p_1^4+p_1^2-4p_1^2p_2-2p_2+8p_1p_3+2p_3-2)\\
&+n\bigl(15\beta^4-4\beta^3(p_1-18)-3\beta^2(2p_1^2+4p_1-4p_2-41)\\
&\qquad-2\beta(2p_1^3+12p_1^2+5p_1-8p_1p_2-24p_2+12p_3-45)\\
&\qquad-p_1^4-4p_1^3-17p_1^2-2p_1+4p_1^2p_2+16p_1p_2+32p_2-8p_1p_3+24\bigr)\\
&+n^2\bigl(-14\beta^3-6\beta^2(p_1+4)+2\beta(11p_1^2-12p_1-20p_2-7)\\
&\qquad-2(p_1^3-8p_1^2+9p_1-4p_1p_2+8p_2+16p_3+2)\bigr)\\
&+n^3\bigl(-36\beta^2+24\beta(p_1-1)+4(3p_1^2+6p_1-8p_2+2)\bigr)+8n^4(p_1-9\beta-6)-32n^5.
\end{align*}
For the extended upside-down identity,
\begin{align*}
\rho_{\rU}(n)={}&\frac{3(n-1)(3n-1)(3n-2)\bigl((2n-1)^2-4(\ell_1+\ell_2)^2\bigr)}{(2n-1)\,\PU(n;\ell)}\\
&\times\frac{(2n-1+2\ell_1-2z)(2n-1+2\ell_2-2z)(1+2z-\ell_1-\ell_2)}{(n-1-2z)(n-2-2z)(n-1+u)(1+2z-2\ell_1-2\ell_2)}.
\end{align*}
The summand of \eqref{eq:thm-upside} and $\rho_{\rU}$ come from the seed \textsf{Seed11} of the package \textsf{WZSeedTools} accompanying \cite{au2026harmonic},
\begin{multline*}
f_{11}(a,b,c,d;m)=(-1)^d(2a+3m)\,\frac{\G(a+m)\G(2b+m)\G(2c+m)}{\G(1+m)\G(1+a-b+m)\G(1+a-c+m)}\\
\times\frac{\G(d+m)\G(\frac12+a-d+m)\G(1+2a-2b-2c+m)}{\G(\frac12+b+c+m)\G(2d+m)\G(1+2a-2d+m)}\\
\times\frac{\G(1+a-b-d)\G(1+a-c-d)\G(\frac12+b+c-d)}{\G(b)\G(c)\G(\frac12+b-d)\G(\frac12+c-d)\G(\frac12+a-b-c)\G(1+a-b-c-d)} .
\end{multline*}
The summand is the WZ mate of $f_{11}$ in the direction $(-2,0,0,-1;1)$, restricted to the hyperplane $a=b+c-\frac12$ with $(b,c,d)=\frac12+(\ell_1,\ell_2,z)$, and $\rho_{\rU}$ comes from the step $d\mapsto d+1$. On this hyperplane the factor $\G(1+2a-2b-2c+m)/\G(\frac12+a-b-c)$ vanishes for $m\ge1$, as $(1+2a-b-c-d-e)_m$ does for the Dougall seed on $H$.

\FloatBarrier

\section*{Use of AI tools}
The questions studied here, and the expectation that cubic, sextic and octic analogues of the Borwein--Bradley identity exist, are the author's. Large language models (Claude Opus~5.5 by Anthropic and GPT-6 by OpenAI) were used throughout this research, under the author's direction, to explore Au's WZ seeds and to find the extensions in Section~\ref{sec:theorems} and their proofs. The certificates in Appendix~\ref{app:tables} were verified by exact computation in Mathematica. The author checked all results and takes full responsibility for the paper.

{\footnotesize
\bibliographystyle{unsrt}
\bibliography{paper_refs}
}

\bigskip
\noindent{\small Einstein Institute of Mathematics, The Hebrew University of Jerusalem, Israel\\
\textit{E-mail}: \texttt{noam.shalev2@mail.huji.ac.il}}

\end{document}